\documentclass[12pt]{amsart}

\usepackage{latexsym, amsmath, amscd, amssymb, amsthm}
\usepackage{bm, mathrsfs, euscript}
\usepackage{hyperref}
\usepackage[alphabetic]{amsrefs}
\usepackage{verbatim}
\usepackage{microtype}

\newtheorem{lemma}{Lemma}[section]
\newtheorem{theorem}[lemma]{Theorem}
\newtheorem{corollary}[lemma]{Corollary}
\newtheorem{proposition}[lemma]{Proposition}
\theoremstyle{definition}
\newtheorem{definition}[lemma]{Definition}
\newtheorem{algorithm}[lemma]{Algorithm}
\newtheorem{example}[lemma]{Example}
\newtheorem{remark}[lemma]{Remark}

\newcommand{\K}{\ensuremath{\Bbbk}} 

\DeclareMathOperator{\inn}{in}
\DeclareMathOperator{\lcm}{lcm}
\DeclareMathOperator{\Spec}{Spec}
\DeclareMathOperator{\val}{val}

\DeclareMathOperator{\trop}{trop}
\DeclareMathOperator{\supp}{supp}
\DeclareMathOperator{\gin}{gin}
\DeclareMathOperator{\id}{id}
\DeclareMathOperator{\lc}{lc}
\DeclareMathOperator{\lm}{lm}
\DeclareMathOperator{\spann}{span}

\newcommand{\PuiseuxC}{\ensuremath{\mathbb C \{\!\{t \} \! \}}}

\begin{document}

\title{Gr\"obner bases over fields with valuations}

\author{Andrew J. Chan}

\address{Mathematics Institute\\ University of Warwick\\
Coventry, CV4 7AL\\ United Kingdom}
\email{andrew.john.chan@gmail.com}

\author{Diane Maclagan}
\address{Mathematics Institute\\ University of Warwick\\
Coventry, CV4 7AL\\ United Kingdom}
\email{D.Maclagan@warwick.ac.uk}


\begin{abstract}
Let $K$ be a field with a valuation and let $S$ be the polynomial ring
$S:= K[x_1, \dots, x_n]$.  We discuss the extension of
Gr\"obner theory to ideals in $S$, taking the valuations of
coefficients into account, and describe the Buchberger algorithm in
this context.  In addition we discuss some implementation and
complexity issues.  The main motivation comes from tropical geometry,
as tropical varieties can be defined using these Gr\"obner bases, but we
also give examples showing that the resulting Gr\"obner bases can be
substantially smaller than traditional Gr\"obner bases.  In the case
$K =\mathbb Q$ with the $p$-adic valuation the algorithms have been
implemented in a Macaulay 2 package.
\end{abstract}

\maketitle

\section{Introduction}

Most work in computational algebraic geometry makes fundamental
use of the theory of Gr\"obner bases.  In this paper we consider
algorithms for a variant of Gr\"obner bases for ideals in a
polynomial ring with coefficients in a valued field.

A major application for these results is in the rapidly growing field
of tropical geometry.  Let $X$ be the subvariety of $\mathbb P^{n-1}$
defined by a homogeneous ideal $I$ in $S=K[x_1,\dots,x_n]$, where $K$
is a field with a valuation $\val \colon K^* \rightarrow \mathbb R$.
Let $X^0$ be the intersection of $X$ with the torus $T \cong
(K^*)^{n-1}$ of $\mathbb P^{n-1}$.  The tropicalization of $X^0$ is
the set of $w \in \mathbb R^n/\mathbb R(1,\dots,1)$ for which the
modified initial ideal $\inn_{w}(I)$ does not contain a monomial.
This has the structure of a polyhedral complex, and many invariants of
$X$ can be recovered from the tropical variety.

Prior computational work in tropical geometry has focused on 
 ideals with coefficients in either $\mathbb Q$ with the trivial
 valuation or $\overline{\mathbb Q(t)}$, as those cases can be treated
 using standard Gr\"obner techniques.  See the software {\tt
   gfan}~\cite{gfan} for details.  Standard Gr\"obner techniques do
 not suffice for the case of $\mathbb Q$ with the $p$-adic valuation
 $\val_p$, which is of increasing interest thanks to the connections
 between tropical geometry and Berkovich spaces; see, for example,
 \cite{BPR}.

Section~\ref{s:GrobnerTheory} explains how the standard Gr\"obner
algorithms need to be modified to handle general valued fields $K$,
such as $(\mathbb Q,\val_p)$.  The main issue is that the standard
normal form algorithm need not terminate.  The solution is to replace
it by a modification of Mora's tangent cone algorithm; the main
contribution of this part of the paper is the suggestion of an
appropriate \'ecart function.
Unlike the standard basis case, we get a strong normal form; see
Remark~\ref{r:normalform}.  In Sections~\ref{s:complexity} and
\ref{s:implementation} we discuss complexity and implementation
issues.  Degree bounds are as for usual Gr\"obner bases
(Theorem~\ref{t:degreebound}).  While the valuations of coefficients
in a reduced Gr\"obner basis cannot be bounded by the valuations of
the original generators (Example~\ref{e:unbounded}), for
coefficients in $(\mathbb Q,\val_p)$ we can bound the valuations of
coefficients in a reduced Gr\"obner basis using the valuations and
absolute values of coefficients of the generators; see
Proposition~\ref{proposition.complexity.valbound}.

A theoretical consequence of these results is a computational proof
that the tropical variety of an ideal only depends on the field
defined by the coefficients of the generators; see
Corollary~\ref{c:tropicaldoesnotdepend}.  We expect these algorithms
to also have applications outside tropical geometry.  In particular,
they can lead to smaller Gr\"obner bases.  In
Section~\ref{s:cardinality} we give a family of ideals in $\mathbb
Q[x_1,x_2,x_3]$ for which the size of the $p$-adic Gr\"obner basis is
constant but the smallest size of a traditional Gr\"obner basis grows
unboundedly.

  The algorithms have been implemented in a package {\tt
    GrobnerValuations}~\cite{GrobnerValuations} for the computational
  algebraic geometry system {\tt Macaulay2}~\cite{M2}, which is
  available from the authors' webpages.  A preliminary implementation
  is also available in {\tt gfan}~\cite{gfan}.  For a different
  approach to this problem, see \cite{MarkwigRen}.

\noindent {\bf Acknowledgments.}  We thank Spencer Backman and Anders Jensen for
comments on an earlier draft of this paper.  Maclagan was partially
supported by EPSRC grant EP/I008071/1.

\section{Gr\"obner Theory}
\label{s:GrobnerTheory}

In this section we generalize the standard Buchberger algorithm for
Gr\"obner bases so that it takes the valuations of coefficients
into account.  The key algorithm is Algorithm~\ref{a:normalForm},
which computes the normal form of a polynomial.

Let $K$ be a field with a valuation $\val \colon K^* \rightarrow
\mathbb R$.  We denote by $R := \{ a \in K : \val(a) \geq 0 \} \cup
\{0\}$ the valuation ring of $K$, by $\mathfrak{m}:= \{ a \in K :
\val(a) >0 \} \cup \{0 \}$ the maximal ideal of the local ring $R$,
and by $\K := R/\mathfrak{m}$ the residue field.  For $a \in R$ we
denote by $\overline{a}$ the image of $a$ in $\K$.  The image of the
valuation map is denoted by $\Gamma$, and is an additive subgroup of
$\mathbb R$.  We assume that there exists a group homomorphism $\phi :
\Gamma \rightarrow K^*$ with $\val(\phi(w))= w$.  This always exists
if $K$ is algebraically closed (see \cite[Lemma
  2.1.15]{TropicalBook}).  For example, the field $K=\mathbb Q(t)$ has
valuation $\val(f/g) = a$ when the Taylor series for $f/g$ is $\alpha
t^a + $ higher order terms.  In this case we can set $\phi(w)=t^w$.
We use the notation $w \mapsto t^w$ for the homomorphism $\phi$ for an
arbitrary field.  We make frequent use of the $p$-adic valuation
$\val_p$ on $\mathbb Q$.  If $a =p^mb/c$, where $p$ does not divide
$b$ or $c$ then $\val_p(a)=m$.  In this case $\Gamma = \mathbb Z$, and
we can take $\phi$ to be $\phi(w)=p^w$.  Another standard, though less
computationally effective, choice of field is the Puiseux series
$\PuiseuxC$ with valuation the lowest exponent occurring.

Let $S$ be the polynomial ring $K[x_1, \dots, x_n]$, and fix a weight
vector $w \in \mathbb R^n$.  For $f = \sum_{u \in \mathbb N^n} c_u x^u
\in S$, let $W:=\trop(f)(w)=\min ( \val(c_u)+ w \cdot u : c_u \neq
0)$.  The initial term of $f$ with respect to $w$ is $$\inn_{w}(f) =
\sum_{\val(c_u)+w \cdot u = W} \overline{t^{-\val(c_u)}c_u} x^u \in
\K[x_1,\dots,x_n].$$ The initial ideal of a homogeneous ideal $I
\subset S$ with respect to $w \in \mathbb R^n$ is $$\inn_{w}(I) = \langle
\inn_{w}(f) : f \in I \rangle \subseteq \K[x_1,\dots,x_n].$$ Note that
$\inn_{w}(I)$ is an ideal in $\K[x_1,\dots, x_n]$.  A finite set
$\mathcal G = \{g_1,\dots, g_s \} \subset I$ is called a {\em
  Gr\"obner basis} for $I$ with respect to $w$ if $\inn_{w}(I) =
\langle \inn_{w}(g_1) , \dots , \inn_{w}(g_s) \rangle$.  The
requirement that the ideal $I$ be homogeneous is not necessary to
define an initial ideal, but is for a Gr\"obner basis to have expected
properties; see Remark~\ref{r:homogeneity}.

This modification of the original definition of a Gr\"obner basis
comes from tropical geometry; see \cite{SpeyerThesis}.  When the
valuation on $K$ is trivial this initial ideal is the standard initial
ideal of $I$ with respect to the weight vector $-w$.  While many
properties of these Gr\"obner bases are well understood when the
valuation is nontrivial, computational issues have not yet been
addressed in the literature.

\begin{example}
Let $f= 3x^2+xy+18y^2 \in \mathbb Q[x,y]$, where $\mathbb Q$ has the
$3$-adic valuation.  For $w = (0,0)$ we have $W = 0$, and $\inn_w(f) =
xy \in \mathbb Z/3\mathbb Z[x,y]$.  For $w=(1,4)$ we have $W=3$, and
$\inn_w(f) = x^2$, and for $w=(2,0)$ we have $W=2$ and $\inn_w(f) =
xy+\overline{3^{-2}18}y^2 = xy+2y^2$.
\end{example}

A Gr\"obner basis for an ideal $I$ can be computed by a modification
of the standard Buchberger algorithm 
as we explain below.  The main difference is in the normal form algorithm
for the remainder of a polynomial on division by a set of other
polynomials.  The difficulty is that a naive implementation of the
normal form algorithm need not terminate, as the following example
shows.

\begin{example} \label{e:nontermination}
Let $K=\mathbb{Q}$ with the $2$-adic valuation.  Consider the standard
normal form algorithm, where the term to be canceled at each stage is
taken to be the term with the lowest valuation of the coefficient.  
Using this to compute the remainder of $x \in \mathbb Q[x,y,z]$ on
division by $\{x-2y,y-2z,z-2x\}$, we reduce $x$ by $x-2y$ to get $2y$.
This is then reduced by $y-2z$ to get $4z$, which in turn is reduced
by $z-2x$ to get $8x$.  This reduction continues indefinitely.
\end{example}

This problem also arises in the theory of standard bases; see for
example \cite[\S 4.3]{CLO2}.  The solution in that setting, Mora's
tangent cone algorithm, is to allow division by previous partial
quotients.  Termination is assured by a descending nonnegative integer
invariant called the \'ecart which measures the difference in degrees
between two possible initial terms of a polynomial.  A difficulty in
generalizing this function to Gr\"obner bases with valuations is that
this difference must take the valuations of the coefficients into
account, so would naturally lie in the not-necessarily-well-ordered
group $\Gamma$.  Even for the valuation $\val_p$ on $\mathbb Q$, where
$\Gamma = \mathbb Z$, the standard \'ecart function does not work
directly.

The following algorithm modifies Mora's algorithm to take into account
the valuations of the coefficients.  It uses a function $E(f,g)$,
which takes two homogeneous polynomials and returns a nonnegative
integer.  In Lemma~\ref{l:normalForm.terminates} we give one option
for this function that ensures termination.  We present the algorithm
with the function $E$ unspecified as more efficient functions $E$ may
exist.

As in all normal form algorithms this is a generalization of long
division, which works by canceling the ``leading term'' of the
polynomial $f$.  An added complication is that we do not assume that
the weight vector $w$ is generic, so the leading term $\inn_w(f)$ is
not necessarily a monomial.  For this reason we also fix an arbitrary
monomial term order $\prec$ (in the sense of usual Gr\"obner theory)
to determine which term of $\inn_w(f)$ to cancel.  If $w$ is
sufficiently generic with respect to the input polynomials $\prec$
will play no role.  For $f \in K[x_1,\dots,x_n]$, 
$\inn_{\prec}(\inn_w(f)) = \alpha x^u$ denotes the leading term, including the
coefficient.  We denote by $\lm(f)$ the monomial $x^u$
occurring in $\inn_{\prec}(\inn_w(f))$, and by $\lc(f)$ the
coefficient of $x^u$ in $f$.  Note that $\lc(f) \in K$, not
$\K$, and that $\lc(f)$ and $\lm(f)$ depend on both $w$ and $\prec$.

We also use the following partial order on polynomials, which plays
the role of comparing initial monomials in usual Gr\"obner bases.

\begin{definition} \label{d:totalorder}
Fix homogeneous polynomials $f,g \in K[x_1,\dots,x_n]$, $w \in \mathbb
R^n$, and a term order $\prec$.  Write $\lm(f)=x^u$, $\lm(g)=x^v$,
$\lc(f)=a$, and $\lc(g)=b$.  Then $f < g$ if $\val(a)+w \cdot u <
\val(b) + w \cdot v$ or $\val(a)+ w \cdot u = \val(b) + w \cdot v$ and
$x^u \succ x^v$.  In addition we set $f < 0$ for all nonzero $f$.
This is consistent with the first part of the definition if we regard
the valuation as a function $\val \colon K \rightarrow \mathbb R \cup
\{ \infty \}$.
\end{definition}

For example, if $\mathbb Q$ has the $2$-adic valuation, $w=(1,2)$ and
$\prec$ is the lexicographic term order with $x_1 \succ x_2$, then $x_1^2
< x_2^2 < x_1^5 < 2x_2^2$.  Note that if $f \geq h$ and
$g \geq h$ then $f \pm g \geq h$.

\begin{algorithm} \label{a:normalForm} 

\noindent {\bf Input: } Homogeneous polynomials $\{g_1,\dots, g_s \}$, a homogeneous polynomial $f$ in
$S=K[x_1,\dots,x_n]$, a weight vector $w\in\mathbb R^n$, and a
term order $\prec$.

\noindent {\bf Output: } Homogeneous polynomials $h_1,\dots,h_s, r \in S$ 
satisfying
$$f=\sum_{i=1}^s h_i g_i +r,$$
where 
$h_ig_i \geq f$ for $1 \leq i \leq s$,
and $r \geq f$.
Write $r=\sum b_v x^v$ with $b_v \in K$.  Then in addition $b_v \neq 0$
implies $x^v$ is not divisible by any $\lm(g_i)$.

We call $r$ a \emph{remainder}, or \emph{normal form}, of dividing $f$
by $\{g_1,\dots,g_s\}$.

\begin{enumerate}

\item {\bf Initialize: } Set $T = \{ g_1,\dots,g_s \},
  h_{10}=\dots=h_{s0}=0, q_0 =f, r_0=0$.  Set $j=0$.

\item {\bf Loop: } While $q_j \neq 0$ do:

\begin{enumerate}
\item {\bf Move to remainder: } If there is no $g \in T$ with
  $\lm(g)$ dividing $\lm(q_j)$, then set
  $r_{j+1}=r_j+\lc(q_j)\lm(q_j)$, $q_{j+1}=q_j -
  \lc(q_j)\lm(q_j)$, and $h_{i{j+1}}=h_{ij}$ for all
  $i$.  Set $T = T \cup \{q_j \}$.

\item {\bf Divide: } Otherwise:  \label{i:divide}
\begin{enumerate}
\item \label{i:Chooseg} Choose $g \in T$ such that $\lm(g)$ divides
  $\lm(q_j)$ with $E(q_j,g)$ minimal among all such choices.  
\item If
$E(q_j,g) >0$ then set $T=T \cup \{q_j\}$. 

\item Since $\lm(g)$ divides $\lm(q_j)$ there is a
  monomial $x^v$ with $\lm(x^vg)=\lm(q_j)$.  Set
  $c_v = \lc(q_j)/\lc(x^vg) \in K$.  Let
  $p=q_j-c_vx^vg$.

\item 
If $g = g_{m}$ for some $1 \leq m \leq s$, then set $q_{j+1}=p$,
$h_{m{j+1}}=h_{mj}+c_vx^v$, $h_{i{j+1}}=h_{ij}$ for
$i \neq m$, and $r_{j+1}=r_j$.

\item  \label{i:v} If $g$ was added to $T$ at some previous iteration of the
  algorithm, so $g=q_m$ for some $m<j$, then set $q_{j+1}=1/(1-c_v)p$,
 $h_{i{j+1}}=1/(1-c_v)(h_{ij}-c_vh_{im})$, and $r_{j+1}=
1/(1-c_v)(r_j-c_vr_m)$.
\end{enumerate}

\item $j=j+1$.

\end{enumerate}

\item {\bf Output: } Output $h_i=h_{ij}$ for $1 \leq i \leq s$, and $r=r_j$.

\end{enumerate}

\end{algorithm}

\begin{example}
Let $f = x^2+y^2+z^2 \in \mathbb Q[x,y,z]$ where $\mathbb Q$ has the
$2$-adic valuation, and let $g_1=y+16z$.  Fix $w = (3,2,1)$, and let
$\prec$ be the lexicographic order with $x \prec y \prec z$.  For
clarity we underline the term of a polynomial $f$ containing
$\lm(f)$.  We do not specify the function $E(f,g)$, assuming
that it is always positive.  Then the algorithm proceeds as follows.

\begin{enumerate}

\item $T=\{\underline{y}+16z\}$, $h_{10}=0$,
$q_0=x^2+y^2+\underline{z^2}$, $r_0=0$, $j=0$.

\item $T=\{\underline{y}+16z, x^2+y^2+\underline{z}^2\}$, $h_{11}=0$,
$q_1=x^2+\underline{y}^2$, $r_1=z^2$, $j=1$.

\item $T=\{\underline{y}+16z, x^2+y^2+\underline{z}^2,
x^2+\underline{y}^2 \}$, $h_{12}=y$, $q_2=\underline{x}^2-16yz$,
$r_2=z^2$, $j=2$.

\item $T = \{\underline{y}+16z, x^2+y^2+\underline{z}^2,
x^2+\underline{y}^2, \underline{x}^2-16yz \}$, $h_{13}=y$,
$q_3=-16yz$, $r_3=x^2+z^2$, $j=3$.

\item $T= \{\underline{y}+16z,x^2+y^2+\underline{z}^2,
x^2+\underline{y}^2, \underline{x}^2-16yz, -16yz \}$,
  $h_{14}=y-16z$, $q_4 = 256z^2$, $r_4=x^2+z^2$, $j=4$.

\item $T= \{\underline{y}+16z,x^2+y^2+\underline{z}^2,
x^2+\underline{y}^2, \underline{x}^2-16yz, -16yz , 256z^2 \}$.  In
this case we divide by $g=x^2+y^2+z^2=q_0$, so $c_v=256$.  Thus
$h_{15}=-1/255(y-16z)$, $q_5=1/255(256x^2+\underline{256y}^2)$,
$r_5=-1/255(x^2+z^2)$, and $j=5$.

\item $T= \{\underline{y}+16z,x^2+y^2+\underline{z}^2,
x^2+\underline{y}^2, \underline{x}^2-16yz, -16yz , -16z^2, 256/255x^2
+ \underline{255/256y}^2 \}$.  Then $g=x^2+y^2=q_1$, so $c_v=256/255$.
 Thus $h_{16}=255(1/255(y-16z))=y-16z$, $q_6=0$,
$r_6=-255(-1/255(x^2+z^2) - 256/255z^2)=x^2+257z^2$, and $j=6$.

\item Output $h_1=y-16z$ and $r=x^2+257z^2$.
\end{enumerate}

Note that $x^2+y^2+z^2 = (y-16z)(y+16z)+x^2+257z^2$ and no term of
$x^2+257z^2$ is divisible by $\lm(y+16z)=y$.

\end{example}

\begin{proof}[Proof of correctness]

We show correctness assuming termination.  

We show that the following properties hold at each stage of the algorithm:
\begin{enumerate}
\item  $f= q_j + \sum_{i=1}^s h_{ij}g_i +r_j$;   \label{enum:equality}
\item $h_{ij}g_i \geq f$; \label{enum:larger} 
\item $r_j \geq f$; \label{enum:order1}
\item No term of $r_j$ is divisible by any $\lm(g_i)$; \label{enum:remainder}
\item  $q_j \geq f$; \label{enum:order2}
\item If $q_{j+1} \neq 0$ then $q_{j+1} >
  q_{j}$.  \label{enum:valuationgrowing}
\end{enumerate}
 These properties all hold at the initialization step by construction.
 We now show they continue to hold after each of the three types of
 iteration step.  We also show that in step \ref{i:v} of the algorithm
 we have $1-c_v \neq 0$.  In all cases, write
 $\lc(q_j)\lm(q_j) = c_jx^{\alpha_j}$.  There are three
 possibilities for the division step, which we consider separately.

{\bf Case 1:} {\em Move to remainder}.  Suppose there is no $g \in T$
with $\lm(g)$ dividing $\lm(q_j)$.  Then the only values that change
are $q_j$ and $r_j$, but we have $q_j+r_j=q_{j+1}+r_{j+1}$ by
construction, so the equality \ref{enum:equality} holds.
Condition~\ref{enum:larger} holds at stage $j+1$ since it held at
stage $j$.  Since properties~\ref{enum:order1} and~\ref{enum:order2}
hold for $j$, property~\ref{enum:order1} holds for $j+1$.  The term
that is added to $r_{j+1}$ is not divisible by any $\lm(g_i)$, so
property~\ref{enum:remainder} still holds.  The term
$c_{j+1}x^{\alpha_{j+1}}$ is a nonleading term of $q_j$, so
property~\ref{enum:valuationgrowing} follows, which also implies
property~\ref{enum:order2}.

{\bf Case 2:} {\em Divide, with $g =g_m$}.  Suppose the chosen $g$
with $\lm(g)$ dividing $\lm(q_j)$ is $g_{m}$ for some $1 \leq m \leq
s$.  Since $q_j+h_{mj}g_{m} = q_{j+1}+h_{mj+1}g_m$ by construction,
the equality~\ref{enum:equality} holds in this case as well.  Since
$h_{mj}g_m \geq f$, and $q_j \geq f$, we have $h_{mj+1}g_m
\geq f$.  As the remainder term does not change
properties~\ref{enum:order1} and~\ref{enum:remainder} still hold.
Since $q_{j+1}=q_j-c_vx^vg_m$, we cancel the leading term of $q_j$, so
all terms of $q_{j+1}$ are the sum of a nonleading term of $q_j$ and a
term of $c_vx^vq_m$ that is larger than $c_jx^{\alpha_j}$.  This
implies that $q_j < q_{j+1}$
(property~\ref{enum:valuationgrowing}), which implies
property~\ref{enum:order2} for $j+1$ as above.

{\bf Case 3: } {\em Divide, with $g = q_m$}.  Finally, we consider the
case that the chosen $g$ with $\lm(g)$ dividing $\lm(q_j)$ is $q_m$
for some $m<j$.  Since all $q_i$ are homogeneous of the same degree,
$x^v=1$ in this setting, and $c_v=c_j/c_m$.  Since
property~\ref{enum:valuationgrowing} holds for all smaller values, we
have $\val(c_j)+w \cdot \alpha_j > \val(c_m)+w \cdot \alpha_m$.  Thus
$x^{\alpha_m}=x^{\alpha_j}$ implies $\val(c_v)>0$, so $1-c_v \neq 0$.

Now $f=q_m +\sum_{i=1}^s h_{im}g_i + r_m$, so $q_{j+1}=
1/(1-c_v)(q_j-c_vq_m)$, which equals $1/(1-c_v)((f-\sum_{i=1}^s
h_{ij}g_i -r_j) - c_v(f-\sum_{i=1}^s h_{im}g_i-r_m))$.  Thus $f =
q_{j+1} + \sum_{i=1}^s 1/(1-c_v)(h_{ij}-c_vh_{im}) g_i +
1/(1-c_v)(r_j-c_vr_m) = q_{j+1} + \sum_{i=1}^s h_{ij+1} g_i +r_{j+1}$.
This is equality~\ref{enum:equality}.

Since $\val(1-c_v)=0$, we have $\val(1/(1-c_v))=0$.  Note the
following property of the order $<$ of
Definition~\ref{d:totalorder}: if $p_1 \geq p_2$ and $c \in K$
satisfies $\val(c) \geq 0$ then $cp_1 \geq p_2$.  Then
properties~\ref{enum:larger} and \ref{enum:order1} for $j+1$ follow
from the analogous properties for $j$ and $m$.  No term in either
$r_j$ or $r_m$ is divisible by any $\lm(g_i)$, so the
same is true for $r_{j+1}$.  Finally $p > q_j$ by construction, so
$q_{j+1} = 1/(1-c_v)p > q_j$ as above, so
properties~\ref{enum:order2} and~\ref{enum:valuationgrowing} also
hold.
\end{proof}

\begin{lemma}\label{l:normalForm.terminates}
For homogeneous polynomials $f, g \in S$ with $f = \sum c_u x^u$ and
$g = \sum b_u x^u$, set $E(f,g) := |\{ u : b_u \neq 0, c_u = 0 \}|$.
Algorithm~\ref{a:normalForm} terminates for this choice of function $E$.  
\end{lemma}

\begin{proof}
There are only a finite number of possible supports $\supp(q_j) = \{u
: c_u \neq 0 \}$ of the polynomials $q_j = \sum c_u x^u$, as they all
have the same degree.  Thus after some step $j$ no new support will
occur, so there will be $q_m \in T_j$ with $\supp(q_m) \subseteq
\supp(q_j)$, and so $E(q_j,q_m)=0$.  Since we remove the leading term
of $q_j$ at the $j$th step, either by moving it to the remainder, or
by canceling it, when $\supp(q_m) \subseteq \supp(q_j)$ we have
$\supp(q_{j+1}) \subsetneq \supp(q_j)$.  Since the size of the support
cannot decrease indefinitely, the algorithm must terminate.
\end{proof}

\begin{remark} \label{r:normalform}
Note that Algorithm~\ref{a:normalForm} gives a strong normal form (no
term of the remainder is divisible by any of the monomials $\{
\lm(g_i) : 1 \leq i \leq s \}$), as opposed to the weak normal form
that occurs in the standard basis case.  This is a consequence of
restricting to homogeneous input; see Remark~\ref{r:homogeneity} for more on
this topic.  See \cite[\S 1.6]{singularbook} for details of normal
forms in the standard basis case.
\end{remark}

\begin{remark} \label{r:Zmodpm}
Algorithm~\ref{a:normalForm} also holds, with the same proof in the
following modified setting.  Let $K=\mathbb Q$ with the $p$-adic
valuation.  The valuation $\val_p$ restricts to a function, which we
also denote by $\val_p$, from $\mathbb Z/p^m\mathbb Z$ to the
semigroup $\{0,1,\dots,m-1 \} \cup \{\infty\}$, where $\infty$ acts as an
absorbing element.  Note that $\val_p(ab) = \val_p(a)+\val_p(b)$ and
$\val_p(a+b) \geq \min(\val_p(a),\val_p(b))$ for $a,b \in \mathbb
Z/p^m \mathbb Z$.  We can then define the partial order $<$ on
polynomials in $\mathbb Z/p^m\mathbb Z[x_1,\dots,x_n]$ in the same way
as in Definition~\ref{d:totalorder}.  Also note that in step \ref{i:v}
of the algorithm, since $1-c_v$ has valuation zero (as shown in the
proof), it is not divisible by $p$, so is a unit in $\mathbb
Z/p^m\mathbb Z$.  This means that the algorithm and its proof go
through in this setting.  This variant is used in
Section~\ref{ss:Zmodpm}.
\end{remark}

As in the usual Gr\"obner setting, we can use the normal form
algorithm to compute a Gr\"obner basis using the Buchberger algorithm.
Let $f, g$ be two polynomials in $K[x_1,\dots,x_n]$. 
  We define the
$S$-polynomial of $f$ and $g$ to be
\begin{equation*}
S(f,g):=\lc(g) \frac{\lcm(\lm(f),\lm(g))}{\lm(f)}f-\lc(f)\frac{\lcm(\lm(f),\lm(g))}{\lm(g)}g.
\end{equation*}

\begin{algorithm}\label{a:Buchberger}

\noindent {\bf Input: } A list $\{ f_1, \dots, f_l \}$ of homogeneous
polynomials in $S$, a weight-vector $w \in \mathbb R^n$, and a term order $\prec$.

\noindent {\bf Output: } A list $\{g_1, \dots, g_s \}$ of homogeneous
polynomials in $I= \langle f_1,\dots,f_l \rangle$ such that $\{
\inn_\prec(\inn_{w}(g_i)) : 1 \leq i \leq s\}$ generates
$\inn_\prec(\inn_{w}(I))$.

\begin{enumerate}
\item Set $\mathcal G = \{ f_1,\dots, f_l \}$.  Set $\mathcal P = \{
(g, g') : g, g' \in \mathcal G \}$.
\item While $\mathcal P \neq \emptyset$:
 \begin{enumerate}
\item Pick $(g,g') \in \mathcal P$.
\item Let $r$ be the normal form on  dividing $S(g,g')$ by $\mathcal G$.
If $r \neq 0$ then set $\mathcal G = \mathcal G \cup \{r\}$, and
$\mathcal P = \mathcal P \cup \{ (r, g) : g \in \mathcal G \}$.
\end{enumerate}
\item Return $\mathcal G$.
\end{enumerate}
\end{algorithm}

The proof of the termination and correctness of this algorithm is
almost exactly the same as the proof for usual Gr\"{o}bner bases,
which can be found for example in~\cite{CLO}.  We indicate below the
necessary changes, which use the following lemma.

\begin{lemma} \label{l:vectors}
Fix $v_1,\dots,v_m \in K^n$, and $\beta_1,\dots,\beta_m \in \mathbb R$.  For
$\lambda \in K^m$, write $s(\lambda) = \min(\val(\lambda_i)+\beta_i)$.
Then for fixed $v \in \spann(v_1,\dots,v_m)$ there is a choice of
$\lambda \in K^m$ with $\sum \lambda_i v_i = v$ that maximizes
$s(\lambda)$ among all such choices.
\end{lemma}
\begin{proof}
We first show that for any $\lambda$ with $\sum \lambda_i v_i = v$
there is a $\lambda'$ with $\sum \lambda'_i v_i = v$, $\{ v_i:
\lambda'_i \neq 0 \}$ linearly independent, and $s(\lambda') \geq
s(\lambda)$.  Indeed, if $\{ v_i : \lambda_i \neq 0 \}$ is linearly
dependent, then there is a nonzero $c \in K^m$ with $\sum c_i v_i = 0$
and $c_i \neq 0$ only when $\lambda_i \neq 0$.  After relabeling we
may assume that $\val(c_1)+\beta_1 = \min(\val(c_i)+\beta_i)$.  Since
this then implies that $\val(c_1) \neq \infty$, we have $c_1 \neq 0$,
so we may rescale so that $c_1 =\lambda_1$.  Let
$\lambda'=\lambda-c$.  Then for every $j$
\begin{align*}
\val(\lambda'_j)+\beta_j & = \val(\lambda_j-c_j)+\beta_j\\ & \geq
\min(\val(\lambda_j),\val(c_j))+\beta_j \\ & = \min(\val(\lambda_j)+\beta_j,
\val(c_j)+\beta_j)\\ & \geq \min(\val(\lambda_j)+\beta_j,
\val(\lambda_1)+\beta_1)\\ & \geq s(\lambda),
\end{align*}
 so $s(\lambda') \geq s(\lambda)$.  Since $\{ i: \lambda'_i \neq 0 \}
 \subsetneq \{ i: \lambda_i \neq 0\}$, after iterating a finite number
 of times $\{v_i : \lambda'_i \neq 0 \}$ is linearly independent.  The
 lemma then follows from the observation that if $\{ v_i : \lambda_i
 \neq 0\}$ is linearly independent, then the $\lambda_i$ are
 determined, so the maximum $s(\lambda)$ is achieved at one of these
 finitely many choices.
\end{proof}

\begin{proof}[Proof of termination and correctness of Algorithm~\ref{a:Buchberger}]
Since at each stage the ideal $\langle \inn_{\prec}(\inn_w(g)) : g \in
\mathcal G \rangle$ strictly increases, termination follows as in the
standard case from the fact that the polynomial ring is Noetherian.

The proof of correctness is also essentially the same as in the
standard case; we include it as it takes essentially the same amount
of space as indicating the changes.  Suppose at the end of the
algorithm $\inn_{\prec}(\inn_w(f)) \not \in \langle
\inn_{\prec}(\inn_w(g_i)) : 1 \leq i \leq s \rangle$ for some
homogeneous $f \in I$.  Since the $f_i$ are contained in $\mathcal G$,
we can write $f = \sum h_i g_i$ for some homogeneous polynomials
$h_i$.  Write $\lm(h_ig_i) = x^{u_i}$.  We may assume that
$\min(\val(\lc(h_ig_i))+ w \cdot u_i)$ is maximal over all choices of
counterexample $f$ and description $f=\sum h_i g_i$.  That a maximum
exists follows from Lemma~\ref{l:vectors} applied to the vector space
$S_{\deg(f)}$, with the $v_i$ all polynomials of the form $x^{u}g_j$
where $x^u$ is a monomial of degree $\deg(f)-\deg(g_j)$, and $\beta_i
= w \cdot u'$ for $\lm(x^ug_j)=x^{u'}$.  After renumbering we may
assume that $\min(\val(\lc(h_ig_i))+w \cdot u_i) = \val(\lc(h_jg_j))+w
\cdot u_j$ for $1 \leq j \leq d$, and that in addition
$x^{u_1}=x^{u_i}$ for $1 \leq i \leq d' \leq d$ with $x^{u_1}$ the
largest $x^{u_i}$ among those $i \leq d$.  We may further assume that
$d'$ is as small as possible among descriptions achieving the maximum.
Since
$\inn_{\prec}(\inn_w(h_ig_i))=\inn_{\prec}(\inn_w(h_i))\inn_{\prec}(\inn_w(g_i))
\in \langle \inn_{\prec}(\inn_w(g_1)),\dots,\inn_{\prec}(\inn_w(g_s))
\rangle$, $x^{u_1} \neq \lm(f)$.  This means that $\lm(\sum_{i=1}^{d'}
h_ig_i) \neq \lm(f)$, so $\val(\sum_{i=1}^{d'} \lc(h_ig_i)) \geq
\min(\val(\lc(h_ig_i)))$, and so in particular $d' \geq 2$.  By
hypothesis we can write $S(g_1,g_2) = \sum_{i=1}^s h'_i g_i$ with
$h'_ig_i \geq S(g_1,g_2)$.  Then
\begin{align*} 
f & = \sum_{i=1}^{s} h_ig_i\\
& = \sum_{i=1}^{s} h_ig_i -
  \frac{\lc(h_1g_1)x^{u_1}}{\lc(g_1)\lc(g_2)\lcm(\lm(g_1),\lm(g_2))}(S(g_1,g_2)
  - \sum_{i=1}^s h'_ig_i)\\
& = (h_1 -  \frac{\lc(h_1g_1)x^{u_1}}{\lc(g_1)\lm(g_1)}+
\frac{\lc(h_1g_2)x^{u_1}}{\lc(g_1)\lc(g_2)\lcm(\lm(g_1),\lm(g_2))}
h'_1)g_1 + \\
& \hspace{1cm} (h_2-\frac{\lc(h_1g_1)x^{u_1}}{\lc(g_1)\lm(g_2)}+
\frac{\lc(h_1g_2)x^{u_1}}{\lc(g_1)\lc(g_2)\lcm(\lm(g_1),\lm(g_2))}
h'_2)g_2+ \\
& \hspace{1cm} \sum_{i=3}^{s} (h_i+\frac{\lc(h_1g_2)x^{u_1}}{\lc(g_1)\lc(g_2)\lcm(\lm(g_1),\lm(g_2))}h'_i)g_i \\
& = \sum_{i=1}^s \tilde{h}_i g_i,
\end{align*}
where $\tilde{h}_i$ is defined to be the polynomial multiplying $g_i$
in the previous line.  By construction $\tilde{h}_1 > h_1$ and
$\tilde{h}_i \geq h_i$ for all $i \geq 2$.  Write $x^{\tilde{u}_i}$
for $\lm(\tilde{h}_i g_i)$.  Thus we have a new expression for $f$
with either $\min(\val(\lc(\tilde{h}_ig_i))+ w \cdot \tilde{u}_i)$
larger or this minimum the same and $d'$ smaller, which contradicts
our assumptions on the respective maximality and minimality of these
quantities.  We thus conclude that $f$ does not exist, so
$\inn_{\prec}(\inn_w(I)) = \langle
\inn_{\prec}(\inn_w(g_1)),\dots,\inn_{\prec}(\inn_w(g_s)) \rangle$ as
required.
\end{proof}

Note that Algorithm~\ref{a:Buchberger} and the proof above also holds
in the variation discussed in Remark~\ref{r:Zmodpm}.

After applying Algorithm \ref{a:Buchberger} we have found a set
$\{g_1,\dots,g_s\} \subset I$ such that $\{\inn_\prec(\inn_{w}(g_i)) :
1 \leq i \leq s\}$ generates $\inn_\prec(\inn_{w}(I))$.  This means
that $\{ \inn_w(g_i) : 1 \leq i \leq s \}$ is a (usual) Gr\"obner
basis for $\inn_w(I)$ with respect to $\prec$, so in particular this
set generates $\inn_w(I)$.  We thus conclude that the set
$\{g_1,\dots,g_s\}$ is a Gr\"obner basis for $I$ with respect to $w$.

This Gr\"obner theory shares many of the properties of standard
Gr\"obner bases:

\begin{enumerate}
\item The Gr\"obner basis $\{g_1,\dots,g_s\}$ generates $I$.  The
  proof here is the standard one: if $f \in I$ then the normal form
  $r$ of $f$ with respect to $\{g_1,\dots,g_s\}$ lies in $I$, but
  $\inn_{\prec}(\inn_w(r)) \not \in \inn_{\prec}(\inn_w(I))$ unless
  $r=0$.

\item For any homogeneous ideal $I$, $w \in \mathbb R^n$, and monomial
  term order $\prec$ there is a unique reduced Gr\"obner basis.  This
  is a Gr\"obner basis $\{g_1,\dots,g_s \}$ with the property that the
  $\inn_{\prec}(\inn_w(g_i))$ minimally generate
  $\inn_{\prec}(\inn_w(I))$, and no monomial in $g_i$ except 
  $\lm(g_i)$ is divisible by any
  $\lm(g_j)$. This follows, as in the standard case,
  from the existence of a strong normal form.  Specifically, if
  $\inn_{\prec}(\inn_w(I)) = \langle x^{u_1},\dots,x^{u_s} \rangle$,
  then let $r_i$ be the remainder on dividing $x^{u_i}$ by any
  Gr\"obner basis for $I$ with respect to $w$ and $\prec$.  Set
  $g_i=x^{u_i}-r_i$.

\item The Hilbert function of the two ideals $I$ and $\inn_w(I)$
  (which live in different polynomial rings) agree.  While this
  follows, as in the standard case, from the existence of a strong
  normal form, there are other proofs; see, for
  example,~\cite[Chapter 2]{SpeyerThesis} or \cite[Corollary 2.4.9]{TropicalBook}.
\end{enumerate}

\begin{remark} \label{r:homogeneity}
We remark that the assumption that the ideal $I$, and the Gr\"obner
basis $\{g_1,\dots,g_s\}$, are homogeneous is necessary for many of
these properties of Gr\"obner bases.  For example, a set $\{
g_1,\dots,g_s\} \subset I$ with $\inn_{w}(I) = \langle
\inn_w(g_1),\dots,\inn_w(g_s) \rangle$ need not generate $I$ if it is
not homogeneous.  A simple example is given by $I= \langle x \rangle
\subseteq \mathbb Q[x]$ with the $2$-adic valuation: for $w=0$ the set
$\{g_1=x+2x^2 \}$ satisfies $\inn_w(I)=\langle x \rangle = \langle
\inn_w(g_1) \rangle$, but $\langle x \rangle \neq \langle x+2x^2
\rangle$.
\end{remark}

This algorithmic approach to these initial ideals also allows a short
computational proof of the following theorem of tropical geometry.
See \cite{TropicalBook} for background definitions.

\begin{corollary} \label{c:tropicaldoesnotdepend}
 Let $K$ be a field with a valuation $\val$ for which there is a
 homomorphism $\phi: \Gamma \rightarrow K^*$ with $\val(\phi(w))=w$.
 Let $L$ be an extension field of $K$ with a valuation that restricts
 to $\val$ on $K$.  Let $Y \subseteq (K^*)^n$, and let $Y_L = Y
 \times_{\Spec(K)} \Spec(L)$.  Then $\trop(Y) = \trop(Y_L)$.
\end{corollary}

\begin{proof}
Let $I \subset K[x_1^{\pm 1},\dots,x_n^{\pm 1}]$ be the ideal of $Y
\subset (K^*)^n$.  Then the ideal of $Y_L$ is $I_L=I L[x_1^{\pm
    1},\dots,x_n^{\pm 1}]$.  Let $J = I \cap K[x_1,\dots,x_n]$, and
$J_L=I_L \cap L[x_1,\dots,x_n]$.  This intersection can be calculated
by a (standard) Gr\"obner computation, so the ideals $J$ and $J_L$
have the same generators: $J_L = J L[x_1,\dots,x_n]$.  The definition
of the initial ideal of an ideal taking the valuation of the
coefficients into account extends naturally to the Laurent polynomial
ring.  By the fundamental theorem of tropical geometry (see, for
example \cite[Theorem 3.2.3]{TropicalBook}) $w \in \mathbb R^n$ lies in
$\trop(Y)$ if and only if $\inn_w(I) = \langle 1 \rangle$, and thus if
and only if $\inn_w(J)$ contains a monomial.  Since $J$ and $J_L$ have
the same generators, Algorithm~\ref{a:Buchberger} implies that
regarding the elements of a Gr\"obner basis for $I$ with respect to
$w$ as living in $L[x_1,\dots,x_n]$ gives a Gr\"obner basis for $I_L$
with respect to $w$.  The residue field $\mathbb{L}$ of $L$ is an
extension field of $\K$, so this means that $\inn_w(I_L) =
\inn_w(I)\mathbb L[x_1,\dots,x_n]$.  An ideal contains a monomial if
and only if the saturation by the product of all the variables is the
unit ideal.  Since this can be decided by a (standard) Gr\"obner basis
computation, this means that $\inn_w(I_L)$ contains a monomial if and
only if $\inn_w(I)$ does.  
This implies that $\trop(Y)=\trop(Y_L)$.
\end{proof}

\section{Complexity}
\label{s:complexity}

Given a bound on the degrees of generators for $I$, it is useful to
have a bound on the degrees of elements in a reduced Gr\"obner basis.
The degree bounds in this context are the same as for usual Gr\"obner
bases~\cite{MoraMoller}, \cite{Dube}, as we show below.  We also give
a bound on the valuations of coefficients occurring in a reduced
Gr\"obner basis when working over $\mathbb{Q}$ with the $p$-adic
valuation.  For the degree bounds we use the formulation of
Dub\'e~\cite{Dube}.

\begin{theorem} \label{t:degreebound}
Let $I = \langle f_1,\dots, f_l \rangle \subset K[x_1,\dots,x_n]$ be a
homogeneous ideal, with $\deg(f_i) \leq d$ for $1 \leq i \leq l$.  Fix
$w \in \mathbb R^n$.  Then there is a Gr\"obner basis $\{g_1, \dots, g_s
\}$ for $I$ with respect to $w$ with $\deg(g_i) \leq 2(d^2/2+d)^{2^{n-2}}$.
\end{theorem}

\begin{proof}
In \cite{Dube} it is shown that if $\deg(f_i) \leq d$ for $ 1\leq i
\leq l$, and $\{g'_1,\dots,g'_s\}$ is a standard homogeneous Gr\"obner
basis with respect to some term order $\prec$, then the degree of each
$g'_i$ is bounded by $2(d^2/2+d)^{2^{n-2}}$.  The proof given actually
shows more: if $M$ is any monomial ideal whose Hilbert function agrees
with that of $I$, then $M$ is generated in degrees at most
$2(d^2/2+d)^{2^{n-2}}$.  Denote by $S_K$ the polynomial ring
$K[x_1,\dots,x_n]$ and by $S_{\K}$ the polynomial ring
$\K[x_1,\dots,x_n]$.  By~\cite[Corollary 2.4.9]{TropicalBook} we have
$\dim_{\K}(S_{\K}/\inn_w(I))_{\delta} = \dim_K (S_K/I)_{\delta}$ for all degrees
$\delta$.  Since the initial ideal $\inn_w(I)$ is again a homogeneous
ideal, all of its monomial initial ideals have the same Hilbert
function, so we have
$$\dim_{\K}(S_{\K}/\inn_{\prec}(\inn_w(I)))_{\delta} =
\dim_{\K}(S_{\K}/\inn_w(I))_{\delta} = \dim_K (S_K/I)_{\delta}.$$ Let $M$ be the
monomial ideal in $S_K$ with the same generators as
$\inn_{\prec}(\inn_w(I)) \subset S_{\K}$.  As the Hilbert function of
a monomial ideal does not depend on the coefficient field, $M$ has the
same Hilbert function as $I$, so by \cite{Dube} $M$ is generated in
degrees at most $2(d^2/2+d)^{2^{n-2}}$.  Choose homogeneous polynomials $\{g_1,\dots,g_s\} \subset
I$ such that $\{\inn_{\prec}(\inn_w(g_1)), \dots,
\inn_{\prec}(\inn_w(g_s)) \}$ is a minimal generating set for
$\inn_{\prec}(\inn_w(I))$.  Then $\inn_w(I) = \langle \inn_w(g_1),
\dots, \inn_w(g_s) \rangle$ so $\{g_1,\dots,g_s\}$ is a Gr\"obner
basis for $I$ with respect to $w$.  Since we have
$\deg(\inn_{\prec}(\inn_w(g_i))) \leq 2(d^2/2+d)^{2^{n-2}}$ by above, we
deduce that $\{g_1,\dots,g_s \}$ is a Gr\"obner basis for $I$ with
respect to $w$ with $\deg(g_i) \leq 2(d^2/2+d)^{2^{n-2}}$ for $1 \leq
i \leq s$ as required.
\end{proof}

\begin{remark}
Fix $w \in \mathbb R^n$, and let $J_w$ be the standard initial ideal of
$I$ with respect to the weight vector $w$ (not taking the valuation
into account).  There exists $v \in \mathbb R^n$ for which
$\inn_{v-\ell w}(I)$ does not not depend on $\ell$ for $\ell > 0$.
Such a $v$ can be chosen from any cell in the Gr\"obner complex of $I$
that has $w$ in its recession cone; see \cite[Theorem
  3.5.6]{TropicalBook}.  We then have 
have $J_{w} = \inn_{v- w}(I)$; the minus sign is because the
initial ideal taking the valuation into account uses min instead of
max.  This means that any usual initial ideal, and thus any usual
Gr\"obner basis, occurs in this setting, so any improvement to
Theorem~\ref{t:degreebound} would also have to improve the bounds of
\cite{MoraMoller} and \cite{Dube}.
\end{remark}

Since the valuations of coefficients also play an important role in
computing these Gr\"obner bases, it is also useful to bound the
valuations that may occur.  This is not possible in full generality,
as the following example shows.

\begin{example}

\label{e:unbounded}
Let $K=\mathbb{Q}(t)$ with the valuation of a rational function given
by taking the lowest exponent occurring in a Taylor series for the
function.  Fix an integer $a \gg 0$ and weight vector
$w=(1,a,2a)$. Let $I$ be the ideal in $K[x,y,z]$ generated by the two
polynomials $f=x+z$ and $g=x^2+(1+t^{a})xz+xy$.  We compute a
Gr\"{o}bner basis by looking at the $S$-polynomial
$S(f,g)=xf-g=-xy-t^{a}xz$. Computing the remainder on division by
$\{f,g\}$ we obtain $yz+t^{a}z^2$ which is a nonzero polynomial with
initial term $yz$.  It is added to the Gr\"{o}bner basis at this stage
by the Buchberger Algorithm (Algorithm~\ref{a:Buchberger}).  Further
running of this algorithm shows that $\{x+z, yz+t^az^2\}$ is a
Gr\"obner basis for $I$.  This can also be seen from applying
Buchberger's criterion (B1); see Section~\ref{s:Buchbergercriteria}.
Notice that we started with polynomials where the valuations of all
the coefficients were zero and we have an element of the reduced
Gr\"{o}bner basis which has a coefficient with valuation $a$ showing
that unbounded valuations may potentially occur when computing
Gr\"{o}bner bases.  The field $K=\mathbb Q(t)$ is only chosen for
concreteness; such an example exists for any nontrivially-valued
field.
\end{example}

When $K=\mathbb{Q}$ with the $p$-adic valuation the valuation of
coefficients that can occur in a reduced Gr\"{o}bner basis can be
bounded in terms of the absolute values of the original coefficients.

Let $I = \langle f_1,\dots,f_l \rangle$ be a homogeneous ideal in
$\mathbb{Q}[x_1,\dots,x_n]$ with $\deg(f_i) \leq \delta$ for $1 \leq
i \leq l$.  Fix $\val$ to be the $p$-adic valuation on $\mathbb Q$.
Write $f_i = \sum c_{u,i}x^u$ where we assume (by clearing
denominators or dividing by a common factor) that $c_{u,i} \in \mathbb
Z$ and that for each $i$ we have $\min_u \val(c_{u,i}) = 0$.

\begin{proposition}\label{proposition.complexity.valbound}
Let $I=\langle f_1,\dots,f_l\rangle$ be a homogeneous ideal in
$\mathbb{Q}[x_1,\dots,x_n]$ with assumptions as above.  Let $C =
\max_{u,i} |c_{u,i}|$.  Fix $w \in \mathbb R^n$.  Then there is a
Gr\"obner basis $\{g_1,\dots,g_s\}$ for $I$ with respect to $w$, with
$g_i = \sum_{u,i} b_{u,i} x^u$, that satisfies $$\val(b_{u,i}) \leq A/2
\log_p (C^2 A),$$ where $A = \dim_{\mathbb Q}(I_D)$ for $D = 2(\delta^2/2+\delta)^{2^{n-2}}$.
\end{proposition}

\begin{proof}
 As the
 Hilbert functions of $I$ and $\inn_{w}(I)$ agree \cite[Corollary
   2.4.9]{TropicalBook} we have that $\dim_{\mathbb Q} I_d =
 \dim_{\mathbb Z/p\mathbb Z} (\inn_{w}(I)_d)$ for all $d$.  Fix a term
 order $\prec$ on $\mathbb Z/p\mathbb Z[x_1,\dots,x_n]$.  Let $H(d) =
 \dim_{\mathbb Q}(I_d)$.

For $d\le D$, form an $H(d) \times\binom{n+d-1}{d}$ matrix $A_d$ with
columns indexed by the monomials of degree $d$ ordered so that those
in $\inn_{\prec}(\inn_{w}(I))_d$ come first.  The rows of $A_d$ are
the coefficients of polynomials forming a $\mathbb Q$-basis for $I_d$;
we may take these polynomials to be monomial multiples of the generators $f_i$, so
all entries of $A_d$ have absolute value at most $C$.

Let the submatrix of $A_d$ indexed by the first $H(d)$ columns be
denoted by $M_d$.  Note that $M_d$ has full rank; if not, since $A_d$
has rank $H(d)$, there would be a vector in the row-space of $A_d$
with its first $H(d)$ entries zero, and thus there would be a non-zero
polynomial $f$ in $I_d$ for which $\inn_{\prec}(\inn_{w}(f))$ does not
lie in $\inn_{\prec}(\inn_{w}(I))$, which is a contradiction.

Set $B_d=M_d^{-1}A_d$.  Note that the first $H(d)$ columns of $B_d$
are an identity matrix, so the minor $\det((B_d)_J)$ of $B_d$ indexed
by the set $J:=(\{1,\dots,H(d)\}\cup\{j\})\backslash \{i\}$ equals
$(-1)^{H(d)-i}(B_d)_{ij}$.  Since $(B_d)_J=M_d^{-1}(A_d)_J$,
\begin{displaymath}
\begin{array}{rcl}
  \val ((B_d)_{ij})&=&   \val (\det((B_d)_J)) \\
  & = &\val(\det(M_d^{-1}(A_d)_J))\\
  & = & -\val(\det(M_d))+\val(\det((A_d)_J)).
\end{array}
\end{displaymath}
Hadamard's inequality (see for example \cite[Corollary
  14.2.1]{Garling}) states that if $M$ is an $N \times N$ matrix with
the absolute value of the entries bounded by $C$, then $|\det(M)| \leq
C^N N^{N/2}$.  Thus $|\det((A_d)_J)| \leq C^{H(d)}H(d)^{H(d)/2}$.
Since $\det(A_d)_J$ is an integer, $\val(\det((A_d)_J) \leq
\log_p(|\det((A_d)_J)|)$.  By construction all entries of $M_d$ have
nonnegative valuation, so $\val(\det(M_d)) \geq 0$.  Thus
$$\val ((B_d)_{ij}))\le \log_p (C^{H(d)} H(d)^{H(d)/2}) = H(d)/2 \log_p (C^2H(d)).$$

By Theorem~\ref{t:degreebound} there is a Gr\"obner basis
$\{g_1,\dots,g_s\}$ for $I$ with respect to $w$ with $\deg(g_i) \leq
D$, which can be chosen so $\{ \inn_{w}(g_1),\dots,\inn_w(g_s) \}$ is
a Gr\"obner basis for $\inn_w(I)$ with respect to $\prec$.  The
construction of the matrix $B_d$ guarantees that if $g_i$ has degree $d$
then the coefficients of $g_i$ form a row of the matrix $B_d$.  This
follows from the fact that there is a unique homogeneous polynomial
$f$ in $I$ with $\inn_{\prec}(\inn_{w}(f))$ equal to a prescribed
monomial $x^u$ with the property that the coefficient of $x^u$ in $f$
is one, and no other term of $f$ lies in $\inn_{\prec}(\inn_w(I))$.
Thus the valuation of the coefficients of $g_i$ is bounded as above.
Since $H(d)$ is an increasing function of $d$, the bound is largest
when $d=D$, so $H(d)=A$, from which we see that the valuations of any
of the coefficients of any $g_i$ is bounded by $A/2 \log_p (C^2A)$ as
required.
\end{proof}

\section{Implementation Issues}
\label{s:implementation}

We focus on  $K=\mathbb{Q}$ with the
$p$-adic valuation.   This has been implemented as a
package in \texttt{Macaulay2}~\cite{GrobnerValuations}.
 As is common for Gr\"obner algorithms with coefficients in $\mathbb
 Q$, a major issue in practical implementations is coefficient
 blow-up.  We found examples where coefficients became so large that
 computations would not terminate within the memory space limitations.
 Thus it was necessary to consider ways to improve the speed and
 efficiency of the algorithms, the two main ways of which are:

\begin{enumerate}
 \item Using criteria to decide a priori that certain $S$-polynomials reduce to zero;
 \item Working over $\mathbb{Z}/p^m\mathbb{Z}$ for some suitably large $m\in\mathbb{N}$.
\end{enumerate}

\subsection{Choice of S-Polynomials - Buchberger's Criteria} \label{s:Buchbergercriteria}

Suppose we are at some intermediate stage of the Buchberger Algorithm
where we have a set $\mathcal P$ of critical pairs still to consider and we
are about to compute the $S$-polynomial of the pair $(f_i,f_j)$.  Then

\begin{description} 
 \item[B1\label{B1}] holds if $\lcm(\lm(f_i),\lm(f_j))=\lm(f_i)\lm(f_j)$; 
 \item[B2\label{B2}] holds if there exists some $k\ne i,j$ such that the pairs $(f_i,f_k)$ and $(f_j,f_k)$ are not in $\mathcal P$ and $\lm(f_k)$ divides $\lcm(\lm(f_i),\lm(f_j))$. 
\end{description}

From the work of Buchberger~\cite{Buchberger} for usual Gr\"obner
bases, if either of these conditions hold then we know a priori that
the $S$-polynomial reduces to zero. The proof can be found for example
in~\cite{CLO}: the proof for \ref{B1} is Proposition 4, and the proof
for \ref{B2} is Proposition 10 of \cite[\S 2.9]{CLO}.  The first proof
follows through verbatim in this situation, while the second requires
the same modifications as in the proof of
Algorithm~\ref{a:Buchberger}.  We illustrate the usefulness of the
criteria with an example.

\begin{example} \label{e:BuchbergerCriterionExample}
Let $K=\mathbb{Q}$ with the $2$-adic valuation and let $S$ be the
polynomial ring $\mathbb{Q}[x_1,\dots,x_9]$. Let $I$ be the ideal
generated by polynomials $\{- 3x_1x_4 + 6x_3x_4 + 3x_1x_5 + 92x_2x_5 +
2x_3x_5 - 23x_2x_6 - 2x_3x_6, x_1x_8 + 7x_2x_8 - 4x_3x_8 - 6x_1x_9 -
3x_2x_9, x_4x_8 + 3x_5x_8 - 3x_6x_8 - 24x_5x_9 - 3x_6x_9, - x_2x_4 -
4x_3x_4 + x_2x_5 + 4x_3x_5 + 23x_2x_6 + 2x_3x_6, - 13x_1x_7 - 4x_3x_7
+ 7x_2x_8 + 28x_3x_8 - 65x_1x_9 - 3x_2x_9 - 32x_3x_9, x_4x_7 +
27x_5x_7 - 9x_6x_8 + 5x_4x_9 + 135x_5x_9 - 9x_6x_9, - 4x_2x_5 -
16x_3x_5 + 3x_1x_6 + x_2x_6 - 2x_3x_6, 13x_2x_7 - 8x_3x_7 + x_2x_8 +
4x_3x_8 + 59x_2x_9 - 64x_3x_9, 8x_5x_7 + x_6x_7 - 3x_6x_8 + 40x_5x_9 +
5x_6x_9, 4x_2x_5x_8 + 16x_3x_5x_8 + 20x_2x_6x_8 - 10x_3x_6x_8 -
24x_2x_5x_9 - 96x_3x_5x_9 - 3x_2x_6x_9 - 12x_3x_6x_9\}$.  This is the
general fiber of a Mustafin variety in the sense of
\cite{MustafinVarieties}.  Its special fiber is the initial ideal with
respect to $w=0$.  

At some intermediate step of the Buchberger Algorithm (Algorithm 
\ref{a:Buchberger}) we compute the normal form of the $S$-polynomial
$6x_3x_4x_6x_7 + 3x_1x_5 x_6x_7 + 24x_1x_4 x_5x_7+ 92x_2x_5x_6x_7 +
2x_3x_5x_6x_7 - 23x_2x_6^2 x_7 - 2x_3x_6^2 x_7 - 9x_1x_4x_6x_8 +
120x_1x_4x_5x_9 + 15x_1x_4x_6x_9$ of the polynomials $- 3x_1x_4 +
6x_3x_4 + 3x_1x_5 + 92x_2x_5 + 2x_3x_5 - 23x_2x_6 - 2x_3x_6$ and
$x_6x_7 +8x_5x_7 - 3x_6x_8 + 40x_5x_9 + 5x_6x_9$. Notice that the
condition B1 holds, so we know a priori that this $S$-polynomial will
reduce to zero, however when we try to compute the normal form, after
a few divisions we obtain a leading coefficient of $1.02624\dots
\times 10^{37,746}$ and after a few more divisions we have exceeded
the memory capabilities of the computer.

 By implementing Buchberger's Criterion, the algorithm no longer considers this critical pair and we compute the Gr\"{o}bner basis to be 
$
\{ 3x_1x_4 - 6x_3x_4 - 3x_1x_5 - 92x_2x_5 - 2x_3x_5 + 23x_2x_6 + 2x_3x_6, 
x_1x_8 + 7x_2x_8 - 4x_3x_8 - 6x_1x_9 - 3x_2x_9, 
x_4x_8 + 3x_5x_8 - 3x_6x_8 - 24x_5x_9 - 3x_6x_9, 
x_2x_4 + 4x_3x_4 - x_2x_5 - 4x_3x_5 - 23x_2x_6 - 2x_3x_6, 
13x_1x_7 + 4x_3x_7 - 7x_2x_8 - 28x_3x_8 + 65x_1x_9 + 3x_2x_9 + 32x_3x_9, 
x_4x_7 + 27x_5x_7 - 9x_6x_8 + 5x_4x_9 + 135x_5x_9 - 9x_6x_9, 
- 4x_2x_5 - 16x_3x_5 + 3x_1x_6 + x_2x_6 - 2x_3x_6, 
13x_2x_7 - 8x_3x_7 + x_2x_8 + 4x_3x_8 + 59x_2x_9 - 64x_3x_9, 
8x_5x_7  - 3x_6x_8 + 40x_5x_9 + 5x_6x_9+ x_6x_7, 
 - 4x_2x_5x_8 - 16x_3x_5x_8 - 20x_2x_6x_8 + 10x_3x_6x_8 + 24x_2x_5x_9  + 3x_2x_6x_9 + 96x_3x_5x_9+ 12x_3x_6x_9  \}
$
\end{example}

\subsection{Working over $\mathbb{Z}/p^m\mathbb{Z}$} \label{ss:Zmodpm}

While it is sometimes unavoidable to get large coefficients when
computing a Gr\"obner basis over $\mathbb Q$, these coefficients do
not always have large $p$-adic valuation.  This motivates working in
$\mathbb Z/p^m\mathbb Z$ via the method suggested in
Remark~\ref{r:Zmodpm}.

This requires the following standard subroutine, which details how to
compute a Gr\"obner basis for $I$ given generators for
$\inn_{\prec}(\inn_w(I))$.  This is the usual linear algebra for
reconstructing Gr\"obner bases as in the non-valuation case; we
include it for completeness.

\begin{algorithm} \label{a:liftingalgo}

\noindent {\bf Input: } Homogeneous generators $\{f_1,\dots,f_l \}$
for an ideal $I \subseteq \mathbb Q[x_1,\dots,x_n]$.  A weight vector $w \in
\mathbb R^n$ and a term order $\prec$.  Generators
$\mathcal I = \{x^{u_1},\dots,x^{u_s} \}$ for $\inn_{\prec}(\inn_w(I))$.

\noindent {\bf Output: } A reduced Gr\"obner basis for $I$ with
respect to $w$ and $\prec$.

\begin{enumerate}

\item $\mathcal G = \emptyset$.

\item For each degree $d$ of a monomial $x^{u_i} \in \mathcal I$ do:

\begin{enumerate}
\item Let $h=\dim_{\mathbb Q} I_d$.  Form the $h \times {n+d-1 \choose d}$
  matrix $A_d$ whose rows are the coefficients of a $\mathbb Q$-basis for
  $I_d$.  The columns of $A_d$ are indexed by the monomials of degree
  $d$, and we assume that the monomials in $\inn_{\prec}(\inn_w(I))_d$
  come first in the ordering.  The rows can be taken to be monomial
  multiples of the $f_i$.
\item Let $B_d$ be the result of multiplying $A_d$ by the inverse of
  the first $h \times h$ submatrix of $A_d$.  This submatrix is
  invertible by the argument of the proof of
  Proposition~\ref{proposition.complexity.valbound}.
\item For each $x^{u_i} \in \mathcal I$ of degree $d$, let $g_i$ be the
  polynomial corresponding to the row of $B_d$ that contains a $1$ in
  the column corresponding to $x^{u_i}$.  Add $g_i$ to $\mathcal G$.
\end{enumerate}

\item Output $\mathcal G$.

\end{enumerate}

\end{algorithm}

\begin{proof}[Proof of correctness of algorithm~\ref{a:liftingalgo}]
The chosen polynomials have the property that no monomial other than
$x^{u_i}$ lies in $\inn_{\prec}(\inn_w(I))$, so
$\inn_{\prec}(\inn_w(g_i))=x^{u_i}$.  Thus the initial ideal $\inn_{\prec}(\inn_w(I))$ equals $\langle \inn_{\prec}(\inn_w(g_1)),\dots,\inn_{\prec}(\inn_w(g_r))
\rangle$, so the output is a reduced Gr\"obner basis as required.
\end{proof} 

We incorporate this into the following algorithm, which computes a
Gr\"obner basis modulo $p^m$ for large $m$.

\begin{algorithm} \label{a:modpm}
\hspace{1cm}\\ \noindent {\bf Input: } A list $\{f_1,\dots,f_l \}$ of homogeneous
polynomials in $\mathbb Q[x_1,\dots,x_n]$, a prime $p$, a
weight-vector $w \in \mathbb R^n$, and a term order $\prec$.

\noindent {\bf Output: } A Gr\"obner basis for $\langle f_1,\dots,f_l \rangle$.  

\begin{enumerate}

\item Let $I = \langle f_1,\dots,f_l \rangle$.  Let $f_i^w =
  f_i(p^{w_1}x_1,\dots,p^{w_n}x_n)$ for $1 \leq i \leq l$.  Clear
  denominators in the $f_i^w$, and saturate the resulting ideal in
  $\mathbb Z[x_1,\dots,x_n]$ by $\langle p \rangle$.  Let
  $\tilde{I}_w$ be the image of this ideal in $\mathbb Z/p^m \mathbb
  Z[x_1,\dots,x_n]$.

\item Compute $\inn_{\prec}(\inn_0(\tilde{I}_w))$ using
  Algorithm~\ref{a:Buchberger}. 

\item Lift the resulting initial ideal to a Gr\"obner basis for $I$
  using Algorithm~\ref{a:liftingalgo}. \label{i:lift}

\end{enumerate}

\end{algorithm}

Note that the fact that Algorithm~\ref{a:Buchberger} does compute
$\inn_{\prec}(\inn_0(\tilde{I}_w))$ follows from Remark~\ref{r:Zmodpm}.
The following lemma shows that for $m$ sufficiently large this initial
ideal equals $\inn_{\prec}(\inn_w(I))$, so
Algorithm~\ref{a:liftingalgo} will terminate with the correct answer.

\begin{lemma}\label{prop.modpm.initialidealsequal}
For $m \gg 0$  Algorithm~\ref{a:modpm} terminates with the correct answer.
\end{lemma}

\begin{proof}
We first show that for $m \gg 0$ we have
$\inn_{\prec}(\inn_0(\tilde{I}_w)) = \inn_{\prec}(\inn_w(I))$.  Note
that if $f = \sum c_u x^u$ with $c_u \in \mathbb Z$ with $\val(c_u) <
m$, then the image $\tilde{f}$ of $f$ in $\mathbb Z/p^m \mathbb
Z[x_1,\dots,x_n]$ satisfies $\inn_{\prec}(\inn_0(\tilde{f}))=
\inn_{\prec}(\inn_0(f))$.  Let $I_w = \langle f_i^w : 1 \leq i \leq l \rangle \subseteq 
\mathbb Q[x_1,\dots,x_n]$, so $\inn_w(I) = \inn_0(I_w)$.  By
Proposition~\ref{proposition.complexity.valbound} there is a bound in
terms of the absolute value of the coefficients of the generators of
$I$ on the maximum valuation that occurs in a reduced Gr\"obner basis.
For $m$ larger than this bound we have $\inn_{\prec}(\inn_w(I))
\subseteq \inn_{\prec}(\inn_0(\tilde{I}_w))$.

For the reverse inclusion, fix $x^u \in
\inn_{\prec}(\inn_0(\tilde{I}_w))$.  Choose $f \in \tilde{I}_w$ with
$\inn_{\prec}(\inn_w(f))=x^u$.  By the definition of $\tilde{I}_w$
there is $g \in I_w$ with $f =\tilde{g}$.  By construction
$\inn_0(g)=\inn_0(f)$, so $x^u=\inn_{\prec}(\inn_0(g)) \in
\inn_{\prec}(\inn_0(I_w)) = \inn_{\prec}(\inn_w(I))$.

In the first step of the algorithm, note that generators of the ideal
obtained by clearing denominators and saturating by $\langle p
\rangle$ generate $I_w \cap \mathbb Z_{\langle p \rangle}
       [x_1,\dots,x_n]$.  Since the image of an ideal $J \subset
       \mathbb Z[x_1,\dots,x_n]$ in $\mathbb Z/p^m \mathbb Z [x_1,\dots,x_n]$
       equals the ideal obtained by first taking the image of $J$ in
       $\mathbb Z_{\langle p \rangle}[x_1,\dots,x_n]$ and then taking
       the image in $\mathbb Z/p^m \mathbb Z[x_1,\dots,x_n]$ (using
       that $\mathbb Z_{\langle p \rangle}/\langle p^m \rangle \cong
       \mathbb Z/p^m \mathbb Z$), $\tilde{I}_w$ is the image of $I_w
       \cap \mathbb Z[x_1,\dots,x_n]$ in $\mathbb Z/p^m\mathbb
       Z[x_1,\dots,x_n]$.  The second step computes
       $\inn_{\prec}(\inn_0(\tilde{I}_w))$ by Remark~\ref{r:Zmodpm}.
       The equality $\inn_{\prec}(\inn_0(\tilde{I}_w)) =
       \inn_{\prec}(\inn_w(I))$ then guarantees that we have the
       correct input for Algorithm~\ref{a:liftingalgo}, so the
       algorithm terminates correctly.
\end{proof}

The bound on $m$ to guarantee that we are in the situation given in
Proposition~\ref{proposition.complexity.valbound}, may be ridiculously
large, and not tight.  If instead one uses an ad hoc choice for $m$,
step~\ref{i:lift} of Algorithm~\ref{a:modpm} will fail if the bound
chosen was too low.  We can thus iterate, repeating the computation
with a larger value of $m$.  This seems often to be the best choice in
practice.

\section{Cardinality}
\label{s:cardinality}

In this  section we give an example which shows that a $p$-adic
Gr\"obner basis may be significantly smaller than any standard
Gr\"obner basis.  This gives another motivation to study such
Gr\"obner bases.

Recall that a monomial ideal $M$ is strongly stable, or Borel fixed,
if for all $x^u \in M$ with $u_j>0$ and $i<j$ we have $x_i/x_j x^u \in
M$.  Our construction requires a special case of the following
elementary lemma.

\begin{lemma} \label{l:Borelfixed}
Fix degrees $d_1,\dots,d_l$, and let $\mathbb P = \prod_{i=1}^l
\mathbb P^{{d_i+n-1 \choose d_i}-1}$ be the parameter space for
sequences of homogeneous polynomials $f_1,\dots,f_l \subset
K[x_1,\dots,x_n]$ of degrees $d_1,\dots,d_l$, where $K$ has
characteristic zero.  Then there is a Zariski-open set $U \subseteq
\mathbb P$ for which if $p \in U$ then the ideal $I= \langle
f_1,\dots,f_l \rangle$ generated by the polynomials corresponding to
$p$ has the property that $\inn_{\prec}(I)$ is strongly stable for all
term orders $\prec$.  There are points in $U$ with any prescribed
valuations.
\end{lemma}

\begin{proof}
Fix a term order $\prec$.  Note that $G = \mathrm{PGL}(n,K)$ acts on
$\mathbb P$ by change of coordinates on each factor.  There is a
nonempty open set $V \subset G \times \mathbb P$ for which
$\inn_{\prec}(gI)$ is constant for all $(g,p) \in V$.  Denote this
  initial ideal by $M_{\prec}$.  The existence of this open set $V$
  follows from the theory of comprehensive Gr\"obner
  bases~\cite{Weispfenning}.  For a fixed $p \in \mathbb P$, there is
  an open set $V' \subset G$ for which the initial ideal
  $\inn_{\prec}(gI)$ equals the generic initial ideal
  $\gin_{\prec}(I)$, which is strongly stable; see for
  example~\cite[Theorem 15.23]{Eisenbud}.  By considering any $p \in
  \mathbb P$ for which there is some $g \in G$ with $(g, p) \in V$, we
  see that the initial ideal $M_{\prec}$ is strongly stable.

Since $V$ is open in $G \times \mathbb P$, the set $U_{\prec}=\{
p \in \mathbb P : (\id,p) \in V \}$ is open in $\mathbb P$,
and $\inn_{\prec}(I) = M_{\prec}$ for all $p \in U_{\prec}$.
The group $G$ acts on $G \times \mathbb P$ by $h \cdot (g, p) =
(gh^{-1}, hp) $.  Note that the set $V \subset
G \times \mathbb P$ is invariant under this action.  This means that
the set $U_{\prec}$ is nonempty, as given any $(g,p) \in V$, we also have $(\id,g^{-1}p) \in V$.    If $M_{\prec} = M_{\prec'}$ for two
different term orders $\prec, \prec'$, then we can take $U_{\prec} =
U_{\prec'}$, as the two term orders agree on the initial terms of a
reduced Gr\"obner basis of any $I=I(p)$ with $p \in U_{\prec}$.
The first part of the lemma then follows from the observation that the
Hilbert functions of all initial ideals $M_{\prec}$ agree and there
are only a finite number of strongly stable ideals with a given
Hilbert function, so there are only a finite number of open sets
$U_{\prec}$ to intersect to obtain an open set $U \subset \mathbb P$ with
$\inn_{\prec}(I)$ strongly stable for any $p \in U$ and any
term order $\prec$.

Since $U \subset \mathbb P$ is open, so is its intersection with an
affine chart $\mathbb A^{\sum_{i=1}^l {d_i+n-1 \choose d_i}-l}$.  This open set 
contains the complement of a hypersurface $V(f)$ where $f \in
K[x_1,\dots,x_N]$ for $N= \sum_{i=1}^l {d_i+n-1 \choose d_i}-l$.  We now
show by induction on $N$ that the valuations of a point outside $V(f)$
can be prescribed.  When $N=1$, $V(f)$ is a finite set, so the base
case follows from the fact that there are infinitely many elements of
$K$ with a given valuation.  Now assume that the claim is true for
smaller $N$, and write $f =gx_1^m+ $ lower order terms, where $g \in
K[x_2,\dots,x_N]$.  Then by induction there is $x'=(x_2,\dots,x_N)$
with $g(x') \neq 0$ and with $\val(x')$ prescribed.  By the base case
there is $x_1$ with prescribed valuation for which the univariate
polynomial $f(x_1,x')$ is nonzero.  Then $(x_1,x') \in U$ is the
desired point.
\end{proof}

The other ingredient needed for the construction is the notion of a
Stanley decomposition for a monomial ideal $M \subseteq
K[x_1,\dots,x_n]$.  For $\sigma \subseteq \{1,\dots,n \}$ and a
monomial $x^u$ we denote by $(x^u,\sigma)$ the set of monomials
$\{x^{u+v} : v_i=0 \text{ for } i \not \in \sigma \}$.  A Stanley
decomposition for $M$ is a union $\{ (x^{u_i},\sigma_i) : 1 \leq i
\leq s\}$ such that every monomial in $M$ lies in a unique set
$(x^{u_i},\sigma_i)$.  The key fact about Stanley decompositions is
that the Hilbert function $\dim_{K} I_t$ of $I$ is the sum
$\sum_{i=1}^s {t-|u_i|+|\sigma_i|-1 \choose |\sigma_i|}$.

\begin{theorem}
Fix an even integer $d=2e$.  Let $I=\langle f,g \rangle \subseteq
\mathbb Q[x_1,x_2,x_3]$ be two generic polynomials of degree $d$ where
every coefficient of $f$ except $x_1^d$ and every coefficient of $g$
except $x_2^{e}x_3^{e}$ has positive $2$-adic valuation, and the
remaining two coefficients have valuation zero.  Then $\inn_0(I) =
\langle x_1^d,x_2^{e}x_3^{e} \rangle$ with the $2$-adic valuation, but
any standard initial ideal $\inn_{\prec}(I)$ has at least $1/2(d+3)$
generators.
\end{theorem}

\begin{proof}
Note first that the existence of $f,g$ satisfying these conditions
follows from Lemma~\ref{l:Borelfixed}, from which it also follows that
every standard initial ideal $\inn_{\prec}(I)$ is Borel-fixed.  That
$\{ f,g \}$ is a $2$-adic Gr\"obner basis for $I$ with respect to $w=0$ follows from
Buchberger's criterion~\ref{B1}.

Fix a term order $\prec$, and let $\inn_{\prec}(I) = \langle
x^{u_1},\dots,x^{u_s} \rangle$.  Write $\{1,2,3\}=\{i_1,i_2,i_3\}$ so that $x_{i_1} \succ x_{i_2} \succ x_{i_3}$.
For $u \in \mathbb N^3$, denote by $m(u)$ the index $m(u) = \max(j :
u_{i_j} \neq 0 ) \in \{1,2,3\}$.  Since $\inn_{\prec}(I)$ is
Borel-fixed, the decomposition $\{ (x^{u_i}, \{i_{m(u_i)},\dots,i_3 \}) : 1
\leq i \leq s \}$ is a Stanley decomposition for $\inn_{\prec}(I)$.
This means that $\dim_{\mathbb Q}(\inn_{\prec}(I)_t) = \sum_{i=1}^s {t - |u_i|
  + 3-m(u_i) \choose 3-m(u_i) }$.  Without loss of generality we may
assume that $x^{u_1} =x_{i_1}^d$, and $m(u_i) \geq 2$ for $i \geq 2$.
Since $I$ is generated in degree $d$, $|u_i| \geq d$ for all $i$.
Since the Hilbert function of $I$ and any initial ideal (standard or
$2$-adic) agree, the fact that the $2$-adic initial ideal of $I$ is $\langle x_1^d,
x_2^{e}x_3^{e} \rangle$ implies that $\dim_{\mathbb Q}(I_t) = 2 {t-d +2 \choose
  2}$ for $d \leq t < 2d$.  Thus for $d \leq t < 2d$ we have
\begin{align*}
2{t-d+2 \choose 2} & = \sum_{i=1}^s {t - |u_i|  + 3-m(u_i) \choose 3-m(u_i) } \\
& \leq {t -d+2 \choose 2} +(s-1){ t - d + 1 \choose 1} \\
\end{align*}
so 
$$1/2(t-d+2)(t-d+1) \leq (s-1)(t-d+1).$$
This means that $s \geq 1/2(d+3)$, as required.
\end{proof}

\end{document}